\documentclass[11pt]{amsart}
    \pdfoutput=1

    \usepackage{graphicx}
    \usepackage[english]{babel}
    \usepackage{graphicx}
    \usepackage{framed}
    \usepackage[normalem]{ulem}
    \usepackage{amsmath}
    \usepackage{dynkin-diagrams}
    \usepackage{amsthm}
    \usepackage{amssymb}
    \usepackage{hyperref} 
    \usepackage[capitalize]{cleveref}
    \usepackage{comment}
    \usepackage{tikz}
    \usetikzlibrary{matrix,calc}
    \usepackage{mathtools}
    \usepackage{tikz-cd}
    \usepackage{amsfonts}
    \usepackage{enumerate}
    \usepackage[utf8]{inputenc}
    \usepackage[top=1 in,bottom=1in, left=1 in, right=1 in]{geometry}
       \usepackage{stmaryrd}
       \usepackage[maxbibnames=99, style=alphabetic, maxalphanames=99, minalphanames=6, 
       sorting=nyt, giveninits=true]{biblatex}
       \usepackage[dvipsnames]{xcolor}

\AtEveryBibitem{
    \clearfield{issn}
  \clearfield{isbn}
  \clearfield{doi}
  \clearlist{location}
  \clearfield{month}
  \clearfield{day}
  \clearfield{series}
  \clearlist{language}
  \clearlist{translator}
  \renewbibmacro{in:}{}
}
\newcommand{\vi}{${\en\sf {(i)}}\;$}
\newcommand{\vii}{${\;\sf {(ii)}}\;$}

    \usepackage[colorinlistoftodos]{todonotes}
    \usepackage[all]{xy}
    \usepackage{mathrsfs}

    \newcommand{\G}{\mathbb{G}} 
    \renewcommand{\t}{{\mathfrak{t}}}
    \newcommand{\fc}{{\mathfrak{c}}}
    
    \newcommand{\SL}{\operatorname{SL}}

    \newcommand{\LG}{\mathfrak{g}}
    
    \newcommand{\LT}{\mathfrak{t}}

\def\ccirc{{{}_{\,{}^{^\circ}}}}
    \newcommand{\beq}{\begin{equation}\label}
\newcommand{\eeq}{\end{equation}}

    \newcommand{\Hom}{\operatorname{Hom}} 
    \newcommand{\End}{\operatorname{End}} 

    \newcommand{\en}{{\enspace}}
    \DeclareMathOperator{\Ker}{\mathrm{Ker}}
    \DeclareMathOperator{\Res}{\mathrm{Res}}
    \renewcommand{\O}{\mathcal{O}} 
    \renewcommand{\\}{\backslash}

   \theoremstyle{plain}
    \newtheorem{Theorem}[equation]{Theorem}
    \newtheorem{Proposition}[equation]{Proposition}
    \newtheorem{Lemma}[equation]{Lemma}
    \newtheorem{Corollary}[equation]{Corollary}
   
    \theoremstyle{definition}

 \newtheorem{Conjecture}[equation]{Conjecture}
    
        \newtheorem{Definition}[equation]{Definition}
       
       \theoremstyle{remark}
       \newtheorem{Remark}[equation]{Remark}

\makeatletter
\renewcommand{\subsection}{\@startsection{subsection}{2}{0pt}{-3ex
plus -1ex minus -0.2ex}{-2mm plus -0pt minus
-2pt}{\normalfont\bfseries}} \makeatother

\numberwithin{equation}{subsection}
\allowdisplaybreaks

    \newcommand{\Frac}{{\operatorname{Frac}}}

    \newcommand{\iso}{{\;\stackrel{_\sim}{\to}\;}}

    \newcommand{\M}{\mathcal{M}} 

    \newcommand{\Symt}{\mathrm{Sym(}\mathfrak{t}\mathrm{)}}

    \newcommand{\LTd}{\LT^{\ast}}

    \newcommand{\into}{\,\hookrightarrow\,}
\newcommand{\Sym}{\mathrm{Sym}}
\newcommand{\LGd}{\mathfrak{g}^*}

\newcommand{\bc}{{{\mathfrak C}}}

\newcommand{\Spec}{\mathrm{Spec}}

\newcommand{\Maps}{\mathrm{Maps}}
\newcommand{\affineClosureOfCotangentBundleofBasicAffineSpace}{\overline{T^*(G/U)}}

\newcommand{\mmod}{\text{-mod}}
\newcommand{\flatGSpacesOverLGd}{\mathrm{Flat}_{\LGd}^G}
\newcommand{\Tpsir}{T^{\psi}(G/\overline{U})}
\newcommand{\bfN}{\mathbf{N}}
\newcommand{\tgd}{\widetilde{\mathfrak{g}^*}}
\newcommand{\tgdreg}{\widetilde{\mathfrak{g}^*_{\mathrm{reg}}}}
\newcommand{\LGdreg}{\LGd_{\mathrm{reg}}}
\newcommand{\LC}{\mathfrak{c}}
\newcommand{\Emod}{E_{\mathrm{mod}}}
\newcommand{\nilHeckeAlgebra}{\mathsf{H}}
\newcommand{\Tpsirt}{T^{\psi}(G/\overline{U})_{\LTd}}
\newcommand{\PGL}{\mathrm{PGL}}
\newcommand{\sset}{\subseteq }
\newcommand{\nil}{{\mathcal{N}il}}
\newcommand{\nh}{{\mathsf{H}}}
\newcommand{\al}{{\alpha}}

\newcommand{\Hilb}{\mathrm{Hilb}}
\newcommand{\SO}{\mathrm{SO}}
\newcommand{\ResFreeClosureW}{\Res^W_{\circ}}
\title{Coordinate ring of the universal centralizer via Demazure operators}
\author{Tom Gannon and Victor Ginzburg}    

\begin{document}
\maketitle
\begin{flushright} {\em To George Lusztig, with admiration}
\end{flushright}
\begin{abstract}
We give a simple description of the coordinate ring of the universal centralizer associated to a simply connected semisimple group. To this end, we prove a general result on the Weil restriction of integral affine schemes $X$ equipped with a dominant morphism to $\mathfrak{t}^*$, the dual of a Cartan subalgebra, and a compatible action of the Weyl group $W$. Specifically, we show that the coordinate ring of the scheme $\mathrm{Res}^W(X)$ of $W$-fixed points of the Weil restriction of $X$ to the categorical quotient $\mathfrak{t}^*\sslash W$ can be obtained from the coordinate ring of $X$ by applying Demazure operators if and only if the scheme $\mathrm{Res}^W(X)$ is integral. 
\end{abstract}
\section{Main results}
\subsection{}\label{Weil Restriction Subsection}
Let $\LT$ be a finite-dimensional vector space over an algebraically closed
field $k$ of characteristic zero, $R\sset\t^*$ a reduced root system,
and $R_+$ a set of positive roots.  Let
${W}$ denote 
the corresponding Weyl group, set  $\mathfrak{c} :=\LTd\sslash W:= \Spec(\O(\LTd)^W)$, and let $q: \LTd \to \LC$ denote the quotient map.

Given an affine scheme $X$ over $\t^*$, there is an affine scheme $\Res X
:= \Res_{q} X$ over $\fc$, the {\em Weil restriction} of $X$ to $\fc$,
equipped with a morphism $c(X): \t^*\times_\fc\Res X \to X$.
We assume that  $X$ is equipped with a $W$-action such that the structure morphism
$X\to \t^*$ is $W$-equivariant.
Then the scheme $\Res X$ acquires a $W$-action along the fibers of the structure morphism
$\Res X \to \fc$. In \cite{KnopAutomorphismsRootSystemsandCompactificationsofHomogeneousVarieties}, Knop considered the scheme $\Res^W(X)$ of $W$-fixed points of $\Res X$, which
represents the functor
\begin{align}
  (\text{Schemes over $\fc$})^{\mathrm{op}} \enspace\to\enspace \text{Sets},\quad
  S \mapsto \Hom_{\t^*}^W(\t^*\times_{\fc} S,\, X),\label{functor}
\end{align}
where $\Hom_{\t^*}^W(\t^*\times_{\fc} S,\, X)$ stands for the set of
$W$-equivariant morphisms of schemes over $\t^*$, and the
$W$-action on $\t^*\times_{\fc} S$ is induced by the one on $\t^*$.
Explicitly, this means that, for  $S$ as above, 
the following composite   is a bijection:
  \begin{equation}\label{fun intro}
  \Maps_{\fc}(S, \Res^WX) \xrightarrow{\ \t^*\times_{\fc}(-) \ }
    \Maps_{\t^*}^W(\t^*\times_{\fc} S,\,\t^*\times_{\fc}\Res^W X) \xrightarrow{\ c(X)\ } 
 \Maps_{\t^*}^W(\t^*\times_{\fc} S,\,X).
  \end{equation}
  
Let $\text{Frac}(\t)$ be the ring of rational functions on $\t^*$,
that is,  the fraction field
of the algebra $\O(\t^*)=\Symt$,
and let $\Delta=\prod_{\alpha\in R_+}\, \check\alpha\,\in\O(\t^*)$
denote the product of positive coroots. We assume that the map $X \to \LTd$ is \textit{dominant} or, equivalently, that the map of rings $\Symt \to \mathcal{O}(X)$ is injective, and that $\mathcal{O}(X)$ is torsion-free as a $\Symt$-module. Given a $W$-module $M$, we write $M^{\mathrm{sign}}$ for the $\mathrm{sign}$-isotypic component of the $W$-action in $M$.
\begin{Definition} For $X$ as above, let
${E}(X)$ be the subalgebra of $\text{Frac}(\t)\otimes_{\O(\t^*)}\O(X)$ generated by $\O(X)=
1\otimes \O(X)$
and all elements of the form
$\frac{f}{\Delta},\enspace f\in \O(X)^{\mathrm{sign}}$,
where $\frac{f}{\Delta}:=\frac{1}{\Delta}\otimes f$. 
\end{Definition}
The algebra ${E}(X)$ is stable under the diagonal $W$-action on $\text{Frac}(\t)\otimes_{\O(\t^*)}\O(X)$; we let ${E}(X)^W$ denote
the $\O(\fc)$-subalgebra of $W$-invariants of $E(X)$. One of our main results identifies, for an integral scheme $X$, the ring of functions on an irreducible component of $\Res^W(X)$ with ${E}(X)^W$; it also gives a criterion for determining whether $\Res^W(X) \cong \Spec(E(X)^W)$:

\begin{Theorem}\label{intro thm} Let $X$ be a reduced and irreducible affine variety equipped with a $W$-action and a $W$-equivariant dominant morphism $X\to \t^*$. Then there is an irreducible component $\Res^W_{\circ} X$ of $\Res^W X$ isomorphic to $\Spec(E(X)^W)$. Moreover, 

\vi The scheme $\Res^W(X)$ is reduced and irreducible if and only if its coordinate ring is torsion-free as a $\Symt^W$-module. 

\vii If $X \to \LTd$ is smooth, then $\Res^W(X)$ is irreducible and smooth over $\LC$.

  \end{Theorem}
 
One can use Theorem \ref{intro thm} to describe $\Res^W_{\circ} X$ as an \textit{affine blow-up} of $X\sslash W$ in the sense of, for example, \cite[Definition 2.1]{MayeuxRicharzRomangyNeronBlowupsandLowDegreeCohomologicalApplications}. More precisely, we have: 
 
 \begin{Corollary}\label{intro corollary for res}
 For $X$ as in Theorem \ref{intro thm}, there is an isomorphism \begin{equation}\label{Oc algebra iso}\O(\Res^W_{\circ} X) \cong \O(X)^W\Big[\frac{j}{\Delta}\Big]_{j \in \O(X)^{\mathrm{sign}}}\end{equation} of $\O(\LC)$-algebras identifying the ring of functions on $\Res^W_{\circ} X$ with the smallest subring of the field of fractions of $\O(X)^W$ containing $\O(X)^W$ together with all elements of the form $\frac{j}{\Delta}$ for $j \in \O(X)^{\mathrm{sign}}$. 
 \end{Corollary}

It is also possible to describe the base change $\LTd \times_{\LC} \Res^W_{\circ} X$ of $\Res^W_{\circ}(X)$ as an affine blow-up of $X$, see Corollary \ref{intro corollary for res expanded}.

\begin{Remark}
Let $X = \Spec(A)$ be a reduced and irreducible affine $W$-scheme equipped with a $W$-equivariant dominant map $\pi: X \to \LTd$. In \cite{BielawskiFoscoloHypertoricVarietiesWHilbertSchemesandCoulombBranches}, the authors construct an integral subscheme $\Hilb_{\pi}^W(X)$ of $\Res^W(X)$, defined as the closure of the subscheme of free orbits in $\Res^W(X)$. In fact, by the definition of $\Hilb_{\pi}^W(X)$ and the construction of $\ResFreeClosureW(X)$ in Section \ref{Construction of Component}, both schemes can be identified with the same irreducible component of $\Res^W(X)$, equipped with the reduced scheme structure. 
\end{Remark}

\begin{Remark}
One important example of $\Res^W(X)$, which was first considered by Knop in \cite{KnopAutomorphismsRootSystemsandCompactificationsofHomogeneousVarieties}, is given by taking $X := \mathbb{T} \times \LTd$ for some irreducible affine algebraic group $\mathbb{T}$ equipped with an action of $W$ compatible with the group structure on $\mathbb{T}$, such as $T$ or $\LT$. For such $X$, one can show, using Theorem \ref{intro thm}(ii), that $\Res^W(X)$ is irreducible. Moreover, using the universal property of $\Res^W(X)$, one can equip it naturally with the structure of a group scheme over $\LTd\sslash W$. For these $X$, Corollary \ref{intro corollary for res} thus gives a description of the ring of functions on the group scheme $\Res^W(X)$.
\end{Remark}

\subsection{The universal centralizer}
To any reductive group $G$ with Lie algebra $\LG$ one associates a canonical abelian group scheme $J$ over $\LG\sslash G$, known as the \textit{group scheme of universal centralizers} \cite{LusztigCoxeterOrbitsandEigenspacesofFrobenius}, \cite{KnopAutomorphismsRootSystemsandCompactificationsofHomogeneousVarieties}, \cite{Ngo}, \cite{DonagiGaitsgoryTheGerbeofHiggsBundles}. This group scheme can be described explicitly as the centralizer of a Kostant section. We fix a $G$-invariant nondegenerate bilinear form on $\LG$ and use it to identify $\LG$ with $\LGd$, and hence $\LG\sslash G$ with $\LC$. 

An important theorem of Knop states that the group scheme $J$ can be described in terms of the Weil restriction construction of Section \ref{Weil Restriction Subsection}. More precisely, it follows from \cite[Theorem 7.8]{KnopAutomorphismsRootSystemsandCompactificationsofHomogeneousVarieties} that there is an isomorphism of group schemes 
\begin{equation}\label{Isomorphism of Ngo} J \xrightarrow{\sim} \Res^W(T^*T)\end{equation} 
provided that $G$ does not contain $\mathrm{SO}_{2n + 1}$ as a direct factor for any $n$, see also \cite[Proposition 2.4.7]{Ngo}, \cite[Theorem 11.6]{DonagiGaitsgoryTheGerbeofHiggsBundles}. Thus, for such $G$, the group scheme $J$ can be described entirely in terms of the action of the Weyl group on a maximal torus of $G$.

Using the isomorphism \labelcref{Isomorphism of Ngo} and the results above, we can give an explicit description of the ring of functions on $J$. To this end, we first observe that, since the projection map $T^*T \to \LTd$ is smooth, Theorem \ref{intro thm}(ii) gives $\Res^W(T^*T) = \Res^W_{\circ}(T^*T)$. Therefore, Corollary \ref{intro corollary for res} implies:

\begin{Corollary}\label{intro cor}
If $G$ does not contain $\mathrm{SO}_{2n + 1}$ as a direct factor, there is an isomorphism \begin{equation}\label{GG Description for OJ}\O(J) \cong \O(T^*T)^W\Big[\frac{i}{\Delta}\Big]_{i \in \O(T^*T)^{\mathrm{sign}}}\end{equation}  of $\O(\LC)$-algebras.
\end{Corollary}

We view Corollary \ref{intro cor} as a strengthening of \cite[Proposition 2.8(b)]{BezrukavnikovFinkelbergMirkovicEquivariantKHomologyofAffineGrassmannianandTodaLatticeNEW2026VERSION} in two ways: \begin{enumerate}
    \item We describe $J$ as an affine blow-up of the scheme $T^*T\sslash W$ itself, rather than as the categorical quotient of an affine blow-up of $T^*T$.
    \item The ideal $I_{\cap}$ of $\O(T^*T)$ used in \cite{BezrukavnikovFinkelbergMirkovicEquivariantKHomologyofAffineGrassmannianandTodaLatticeNEW2026VERSION} (and defined below) is likely to be larger than the ideal $I_{\mathrm{sign}}$ generated by $\O(T^*T)^{\mathrm{sign}}$.
\end{enumerate}We now explain these points in more detail. Recall that \cite[Proposition 2.8(b)]{BezrukavnikovFinkelbergMirkovicEquivariantKHomologyofAffineGrassmannianandTodaLatticeNEW2026VERSION} gives an $\O(\LC)$-algebra isomorphism \begin{equation}\O(J) \cong (\O(T^*T)\Big[\frac{i}{\Delta}\Big]_{i \in I_{\cap}})^W\end{equation} where $I_{\cap} := \bigcap_{\alpha}(\alpha - 1, \alpha^{\vee})$,\footnote{Here, we slightly abuse notation and view $\alpha$ as a function on $T \times \mathfrak{t}^*$ via projection onto the first factor and view $\alpha^{\vee}$ as a function on $\mathfrak{t}^*$ via projection onto the second factor.} see also \cite[Theorem 3.21]{GorskyKivinenOblomkovTheAffineSpringerFiberSheafCorrespondence}. 
On the other hand, using Corollary \ref{intro cor}, it is not difficult to show that there is an isomorphism of $\O(\LC)$-algebras \begin{equation}
\O(J) \cong (\O(T^*T)\Big[\frac{i}{\Delta}\Big]_{i \in I_{\mathrm{sign}}})^W
\end{equation} if $G$ does not contain $\SO_{2n + 1}$ as a direct factor; for example, this follows by taking the $W$-invariants of the isomorphism \labelcref{Otstar algebra iso} of Corollary \ref{intro corollary for res expanded} below.

An elementary argument shows that the ideal $I_{\mathrm{sign}}$  is contained in $I_{\cap}$,\footnote{Indeed, since $I_{\cap}$ is an ideal and $I_{\mathrm{sign}}$ is generated by the anti-invariants, it suffices to prove that $\O(T^*T)^{\mathrm{sign}} \subseteq I_{\cap}$. Fix a root $\alpha$, set $I_{\alpha} := (\alpha - 1, \alpha^{\vee})$, and choose some $p_0 \in \O(T^*T)^{\mathrm{sign}}$. We may write $p_0 = p - s_{\alpha}(p)$ for some $p \in \O(T^*T)$. We may further write $p$ as a sum of terms of the form $p_1p_2$ where $p_1 \in \O(T)$ and $p_2 \in \O(\LTd)$. Since \[p_1p_2 - s_{\alpha}(p_1p_2) = p_1(p_2 - s_{\alpha}(p_2)) + s_{\alpha}(p_2)(p_1 - s_{\alpha}(p_1))\] and $(p_2 - s_{\alpha}(p_2)) \in \Symt$ is divisible by $\alpha^{\vee} \in \Symt$ and 
$(p_1 - s_{\alpha}(p_1)) \in \O(T)$ is divisible by $1 - \alpha^{-1} = \alpha^{-1}(\alpha - 1)$, we deduce that $p_0 \in I_{\alpha}$ for all $\alpha$. Therefore $p_0$ lies in the intersection $I_{\cap}$.} but these ideals are \textit{not} expected to be equal for general reductive (or general simply connected) $G$; for example, the \lq rational\rq{} analogues of these ideals do not agree when $G$ has type $B_3$, see \cite{KivinenALieTheoreticGeneralizationOfSomeHilbertSchemes}. Moreover, the only known proof in type $A$ that $I_{\mathrm{sign}} = I_{\cap}$ relies on deep work of Haiman \cite{HaimanHilbertSchemesPolygraphsandtheMacdonaldPositivityConjecture}. 
Therefore, the generating set in \labelcref{GG Description for OJ} is likely smaller than that of \cite[Proposition 2.8(b)]{BezrukavnikovFinkelbergMirkovicEquivariantKHomologyofAffineGrassmannianandTodaLatticeNEW2026VERSION} for groups $G$ that do not contain $\SO_{2n + 1}$ as a direct factor. (On the other hand, the description of Corollary \ref{intro cor} does not hold when $G = \PGL_2 = \SO_3$, whereas the description of \cite[Proposition 2.8(b)]{BezrukavnikovFinkelbergMirkovicEquivariantKHomologyofAffineGrassmannianandTodaLatticeNEW2026VERSION} applies to arbitrary reductive groups.)

\begin{Remark}
If $G$ is an arbitrary semisimple group, then the group scheme $J$ of universal centralizers for $G$ can be described as the categorical quotient $\tilde{J}\sslash Z$, where $\tilde{J}$ is the group scheme of universal centralizers for the simply connected cover $\tilde{G}$ of $G$, and $Z$ is the kernel of the map $\tilde{G} \to G$. The isomorphism \labelcref{GG Description for OJ} for the group $\tilde{G}$ can be checked to be $Z$-equivariant, so Corollary \ref{intro cor} gives an $\O(\LC)$-algebra isomorphism \begin{equation}\O(J) \cong (\O(T^*\widetilde{T})^W\Big[\frac{i}{\Delta}\Big]_{i \in \O(T^*\widetilde{T})^{\mathrm{sign}}})^Z\end{equation} for an arbitrary semisimple $G$, where $\widetilde{T}$ is the preimage of $T$ under the quotient map $\tilde{G} \to G$. More generally, working with the group scheme $\widetilde{G_{\mathrm{ss}}} \times Z(G)^\circ$, one can use Corollary \ref{intro cor} to describe the ring of functions on $J$ in terms of the maximal torus of $\widetilde{G_{\mathrm{ss}}} \times Z(G)^\circ$ for an arbitrary reductive group $G$.
\end{Remark}

\subsection{\texorpdfstring{The nil-Hecke algebra}{The nil-Hecke algebra}}
Following Demazure \cite{DemazureInvariantsSymmetriquesEntiersDesGroupsDeWeylEtTorsion}, Bernstein-Gelfand-Gelfand \cite{BernsteinGelfandGelfandSchubertCellsandCohomologyofGModP}, and Kostant-Kumar \cite{KostantKumarTheNilHeckeRingandtheCohomologyofGModPforaKacMoodyGroupG}, one
defines a $k$-algebra $\nil({W})$ to be the vector space with $k$-basis
$D_w,\, w\in{W}$,  and multiplication table:
\[
  D_y  D_w=\begin{cases} D_{yw} & \text{ if }  \ell(yw)=\ell(y)+\ell(w);\\
                0 & \text{ if }  \ell(yw) <\ell(y)+\ell(w).
               \end{cases} \label{rel1}
             \]


             Kostant and Kumar defined the  {\em nil-Hecke algebra} $\nh$,
which is isomorphic
             to $\Symt \otimes \nil({W})$ as a vector space and in which
             $\Symt \otimes 1$ and
             $1\otimes \nil({W})$ are subalgebras. Multiplication in
             $\nh$  is determined
             by the commutation relation:
\beq{com-rel}
D_{s_\al}\cdot s_\al(\gamma)-\gamma\cdot D_{s_\al}=-\langle \alpha,\gamma\rangle,
 \quad \forall \gamma\in\t,\, \al\in\Sigma,
 \eeq
 where $\Sigma\sset R_+$ is the set of simple roots.
(cf. also
             \cite{BernsteinGelfandGelfandSchubertCellsandCohomologyofGModP}, \cite{DemazureInvariantsSymmetriquesEntiersDesGroupsDeWeylEtTorsion}, \cite{HolmSjamaarTorsionadnAbelianizationInEquivariantCohomology}.)
             Let $M$ be a $W\ltimes {\Symt}$-module which is torsion-free as a ${\Symt}$-module.
It is known (see for example Proposition \ref{Forgetful Functor form nilHeckeAlgebra to S rtimes W mod fully faithful}) that the $W\ltimes {\Symt}$-action on $M$ has a unique extension
to an $\nh$-action on $\text{Frac}(\t)\otimes_{\Symt} M$ such that, for all {\em simple
  reflections} $s_\alpha\in W$, the element ${D}_{s_\al}$
acts as the operator $\text{Frac}(\t)\otimes_{\Symt} M\to \text{Frac}(\t)\otimes_{\Symt} M,\
u\mapsto \frac{1}{{\check\al}}(u-s_\al(u))$.
For a general $w \in W$, the  $D_w$-action is
given by the operator $D_w := D_{s_1}\cdots D_{s_\ell}$,
which is a $\Symt^W$-linear map called a {\em Demazure operator}.
Here $w = s_1\cdots s_{\ell}$ is a reduced expression; by \cite[Theorem 1]{DemazureInvariantsSymmetriquesEntiersDesGroupsDeWeylEtTorsion}, $D_w$ is independent of this choice.

Let $w_0\in{W}$ be the longest element.
The corresponding Demazure operator is given by
\begin{equation}\label{Explicit Formula for Dw0}
  D_{w_0}:\ u\mapsto \frac{1}{\Delta}\sum_{w\in{W}} (-1)^{\ell(w)} w(u),
 \end{equation} see \cite[Proposition 3(b)]{DemazureInvariantsSymmetriquesEntiersDesGroupsDeWeylEtTorsion}.

\begin{Definition}\label{Demazure envelope of algebra definition}
Let $A$ be a commutative $\Symt$-algebra equipped with a $W$-action for which the algebra map $\Symt \to A$ is injective and $W$-equivariant. Assume further that $A$ is torsion-free as a $\Symt$-module. The {\em  Demazure envelope} of $A$ is the smallest subring $E(A)$ of $\Frac(\t)\otimes_{\Symt} A$   which contains $A$ and is closed under the Demazure operators.
        \end{Definition}

If $A$ in Definition \ref{Demazure envelope of algebra definition} is moreover an integral domain, we can (and will) view $E(A)$ as a subring of the field of fractions of $A$. 

We will prove the following Proposition on the Demazure envelope; the first part is proved in Section \ref{ssec env}, and the second part is restated and proved as Lemma \ref{Demazure Envelope Has Adjunction Property for Torsion Free Schemes Algebra Version}. 

    \begin{Proposition}\label{env prop} For $A$ as in Definition \ref{Demazure envelope of algebra definition}, we have:
    
   \vi    The algebra $E(A)$ equals the
      subalgebra of $\Frac(\t)\otimes_{\Symt} A$
      generated by $A$ and the elements $D_{w_0}(a),\, a\in A$.

      \vii If $B$ is a commutative $\Symt^W$-algebra which is torsion-free as a $\Symt^W$-module, then any $W$-equivariant  $\Symt$-algebra homomorphism
      $A\to \Symt\otimes_{\Symt^W}B$ has a unique extension
      to  a $W$-equivariant  $\Symt$-algebra homomorphism
      $E(A)\to  \Symt\otimes_{\Symt^W}B$.
    \end{Proposition}

\subsection{\texorpdfstring{The Whittaker cotangent bundle as a Weil restriction}{The Whittaker cotangent bundle as a Weil restriction}}\label{Whittaker Cotangent Bundle as Weil Restriction Subsection} In Section \ref{Whittaker Cotangent Bundle as Weil Restriction Subsection}, we assume that the derived subgroup of $G$ is simply connected. We also choose a Borel subgroup $B$ and a maximal torus $T$ contained in $B$, let $U := [B, B]$ and let $\overline{U} := [\overline{B}, \overline{B}]$, where $\overline{B}$ is the Borel subgroup containing $T$ opposite to $B$. 
We further choose a nondegenerate character $\psi: \overline{\mathfrak{u}} := \mathrm{Lie}(\overline{U}) \to \G_a$. Finally, for any scheme $S$ over $\LC$, we denote $S_{\LTd} := S \times_{\LC} \LTd$, which we view as a $W$-scheme where $W$ acts on the second factor.

We define the \textit{Whittaker cotangent bundle} of $G$ to be the variety \[\Tpsir := (T^*G \times_{\overline{\mathfrak{u}}^*} \{\psi\})/\overline{U}\] over $\LC$. Using the theory of Kostant sections, one can show (see \cite[Lemma 3.2.3]{GinzburgKazhdanDifferentialOperatorsOnBasicAffineSpaceandtheGelfandGraevAction}) that there is a $G$-equivariant isomorphism $\Tpsir \cong G \times \LC$ over $\LC$. One of the main constructions of \cite{GinzburgKazhdanDifferentialOperatorsOnBasicAffineSpaceandtheGelfandGraevAction} (see also \cite{GannonTheCotangentBundleofGModUPandKostantWhittakerDescent}) gives a birational map \begin{equation}\label{Birational Map For Cotangent Bundles}\varpi: \Tpsirt \to \affineClosureOfCotangentBundleofBasicAffineSpace\end{equation} where $\affineClosureOfCotangentBundleofBasicAffineSpace$ is the affinization of the cotangent bundle of the base affine space $G/U$. Moreover, this map is $W$-equivariant, where $W$ acts on $\Tpsirt$ via its action on the second factor and acts on $\affineClosureOfCotangentBundleofBasicAffineSpace$ via the (quasi-classical) \textit{Gelfand-Graev action}, reviewed, for example, in \cite{GinzburgRicheDifferentialOperatorsOnBasicAffineSpaceandtheAffineGrassmannian}. 

Based on the Coulomb branch descriptions of $\affineClosureOfCotangentBundleofBasicAffineSpace$ and $\Tpsir$ in type $A$ \cite{GannonWilliamsDifferentialOperatorsOnBaseAffineSpaceofSLnandQuantizedCoulombBranches}, \cite{GannonWebsterFunctorialityOfCoulombBranches}, \cite{BravermanFinkelbergNakajimaRingObjectsIntheEquivariantDerivedSatakeCategoryArisingFromCoulombBranches}, as well as on other computations we have performed in that type, we propose the following Conjecture, which is a stronger version of the conjecture posed in the first version of this paper: 

\begin{Conjecture}\label{Our conjecture}
  If $[G, G]$ is simply connected, then there is an isomorphism \[\Res^W(\affineClosureOfCotangentBundleofBasicAffineSpace) \cong \Tpsir\] of schemes over $\LC$. Equivalently, the Demazure envelope of the ring of functions on $\affineClosureOfCotangentBundleofBasicAffineSpace$ is isomorphic to the ring of functions on $\Tpsirt \cong G \times \LTd$ as $W$-equivariant $\O(\LTd)$-algebras.
\end{Conjecture}

In Section \ref{Weil Restriction and Cotangent Bundle of G mod U Section}, we give some evidence for Conjecture \ref{Our conjecture}. More specifically, we will prove:

\begin{Theorem}\label{Representability for Flat G Spaces Over LGd}
If $[G, G]$ is simply connected, then the map \labelcref{Birational Map For Cotangent Bundles} induces an isomorphism \[\Maps_{\LGd}^G(S, \Tpsir) \xrightarrow{\sim}  \Maps_{\LGd \times_{\mathfrak{c}} \LTd}^{G \times W}(S_{\LTd}, \affineClosureOfCotangentBundleofBasicAffineSpace)\] for any affine $G$-scheme $S$ equipped with a compatible \textit{flat} map to $\LGd$.
\end{Theorem}
\subsection{Outline} In Section \ref{Demazure Envelope Section}, we prove results on the Demazure envelope of certain $\Symt$-algebras which will be used in the proof of Theorem \ref{intro thm}. That theorem is proved in Section \ref{Invariant Weil Restrictions Subsubsection}. Finally, in Section \ref{Weil Restriction and Cotangent Bundle of G mod U Section}, we prove Theorem \ref{Representability for Flat G Spaces Over LGd}. 

\subsection{Acknowledgements} We would like to thank Roman Bezrukavnikov, Roger Bielawski, C\'edric Bonnaf\'e, Oscar Kivinen, Mark Haiman, Colleen Robichaux, Burt Totaro, and Ben Webster for interesting and useful discussions. We are grateful to Friedrich Knop and Yiannis Sakellaridis for comments
on an earlier draft of the paper, and would like to thank the anonymous referee for a careful reading of this manuscript.

Our work is dedicated to George Lusztig. Without his revolutionary insights, representation
theory would not exist in its present form. In particular, as mentioned above, the
(group-theoretic variant $J^G_G\to G\sslash G$ of the) group scheme of universal
centralizers first appeared in \cite{LusztigCoxeterOrbitsandEigenspacesofFrobenius}. Lusztig also introduced in
\cite{LusztigEquivariantK} certain
$q$-analogues of Demazure operators,  nowadays called
{\em Demazure-Lusztig operators}. In that paper, Lusztig recognized the true meaning of
the parameter $q$  as a generator of $K^{\mathbb{C}^*}(pt)$ even though the 
variety considered there was the flag variety equipped
with a {\em trivial} $\mathbb{C}^*$-action. This $K$-theoretic interpretation of 
Demazure-Lusztig operators played a major role in the
proof of the Deligne-Langlands conjecture
by Kazhdan and Lusztig \cite{KazhdanLusztigProofOfDL}.

Artificial intelligence (more specifically, ChatGPT 6.1 Astra) was used for translating \cite{DemazureInvariantsSymmetriquesEntiersDesGroupsDeWeylEtTorsion}, for correcting typos and minor errors, and for correcting some more serious errors in Section \ref{Weil Restriction and Cotangent Bundle of G mod U Section} in the first draft of this paper, and for nothing else. 

\section{Demazure envelopes}\label{Demazure Envelope Section}
\subsection{The nil-Hecke ring}
The natural action of $W \ltimes {\Symt}$ on ${\Symt}$ extends
to an $\nh$-action on $\text{Frac}(\t)$. Moreover, for any $\nh$-module $M$, the inclusion $W \ltimes {\Symt} \subseteq \nh$ allows us to view $M$ as a module for the ring $W \ltimes {\Symt}$; in particular, any such $M$ acquires the structure of a $W$-representation. We will use this representation implicitly in what follows. 
We will also use the {\em twisted Leibniz rule}
     \begin{equation}\label{Twisted Leibniz rule}
    D_s\cdot  p= D_s(p) + s(p)\cdot D_s,\quad\forall p \in \Symt  \text{ and all simple reflections $s$}
  \end{equation}
which follows from the commutation relation (\ref{com-rel}). We will use the twisted Leibniz rule to prove the first part of the following Proposition:

\begin{Proposition}\label{standard w0}
  \begin{enumerate}
    \item For any $\nh$-module $M$, the action of $D_{w_0}$ induces a
  $\Symt^W$-linear isomorphism \[D_{w_0}: M^{\mathrm{sign}} \to M^W,\]
  where $M^{\mathrm{sign}}$ denotes the sign isotypic component
  of $M$.
  Its inverse is given by $m \mapsto \frac{1}{|W|}\Delta m$.
  Furthermore, the natural inclusions \[M^W \subseteq \bigcap_{w \in W\setminus\{e\}}
    \Ker(D_w) \subseteq \bigcap_{s}\Ker(D_s)\] are equalities, where $s$ varies over simple reflections.

\item  The algebra  $\nilHeckeAlgebra$ is free as a right $\Symt$-module with
  basis $D_w,\,w \in W$.

\item  The action of $\nh$ on ${\Symt}$ induces an isomorphism
  $\nh\iso \End_{\Symt^W}(\Symt)$ of $\Symt^W$-algebras.
  \end{enumerate}
\end{Proposition}
\begin{proof}
Observe that, since $D_{s}D_{w_0} = 0$ for any simple reflection $s=s_\alpha$, we have that \[(1 - s_\alpha)D_{w_0}(m) = \alpha^{\vee}D_sD_{w_0}(m) = 0\] for any $m \in M$. This shows that $D_{w_0}$ maps $M$ (and, in particular, the sign isotypic component of $M$) to the trivial isotypic component and that $\bigcap_{s}\Ker(D_s) \subseteq M^W$. Multiplication by $\Delta$ clearly maps $M^W$ to the sign isotypic component. By repeated applications of the twisted Leibniz rule (\ref{Twisted Leibniz rule}), we deduce that $D_{w_0}(\Delta m) = D_{w_0}(\Delta)m$ if $m \in M^W$. Therefore, by (\ref{Explicit Formula for Dw0}), \[\frac{1}{|W|}D_{w_0}(\Delta m) = m\] for $m \in M^W$. Conversely, if $m \in M^{\mathrm{sign}}$, then \[\frac{1}{|W|}\Delta D_{w_0}(m) = \frac{1}{|W|}\sum_w(-1)^{\ell(w)}w(m) = \frac{1}{|W|}|W|m = m\] as desired.

To prove (2), one uses 
the $\Symt$-antilinear antipode map constructed in \cite[Chapter 11]{KumarKacMoodyGroupsTheirFlagVarietiesandRepresentationTheory}.
Since the collection $D_w,\,w\in W$ is a basis of $\nh$ as a left $\Symt$-module,
it follows that this collection gives a basis of $\nh$ as a right $\Symt$-module as well.
Part (3) is proved in \cite[Lemma 7.1.5]{Gin}.
\end{proof}

\subsection{\texorpdfstring{Demazure envelopes}{Demazure envelopes}}\label{ssec env} 

  \begin{Definition}\label{Demazure envelope of module definition}
      Let $M$ be a torsion-free $\Symt$-module
         with a compatible $W$-action.
    The (module) \textit{Demazure envelope} $\Emod(M)$ of $M$ is the smallest $\Symt \rtimes W$-submodule of $\Frac(\t)\otimes_{\Symt} M$ containing $M$
    and closed under the Demazure operators.
    \end{Definition}

%

We denote the category of modules for the ring $\Symt \rtimes W$ by $\Symt\mmod^W$. In what follows, we will use the following elementary observation on the relationship between $\Symt^W\mmod$ and $\Symt\mmod^W$:

\begin{Proposition}\label{Fully faithfulness of tensoring up even for algebra objects}
    The functor $\Symt \otimes_{\Symt^W} -: \Symt^W\mmod \to \Symt\mmod^W$ is fully faithful. Moreover, the induced functor on commutative algebra objects is fully faithful. (In other words, if $S_1$ and $S_2$ are affine schemes over $\LTd\sslash W$, the map \begin{equation*}\label{Map on Hom Sets for Affine schemes over lTdsslashW}\Hom_{\LTd\sslash W}(S_1, S_2) \xrightarrow{(-) \times_{\LTd\sslash W} \LTd} \Hom_{\LTd}^W(S_1^{\LTd}, S_2^{\LTd})\end{equation*} is bijective.)
    
    In either case, the essential image consists of those $M \in \Symt\mmod^W$ such that the multiplication map $c(M): \Symt \otimes_{\Symt^W} M^W \to M$ is an isomorphism. 
\end{Proposition}

\begin{proof}
In either case, the functor $L := \Symt \otimes_{\Symt^W} -$ has a right adjoint given by $R := (-)^W$: indeed, the unit map $M_0 \to (\Symt \otimes_{\Symt^W} M_0)^W$ for a fixed $M_0 \in \Symt^W\mmod$ is $m_0 \mapsto 1 \otimes m_0$, and, for $M \in \Symt \mmod^W$, the counit map $c(M): \Symt \otimes_{\Symt^W} M^W \to M$ is given by the formula $p \otimes m \mapsto pm$. (One can check directly that these satisfy the relevant adjunction criteria.) Since $(-)^W$ is exact, the canonical map \[\Symt^W \otimes_{\Symt^W} M_0 \to (\Symt \otimes_{\Symt^W} M_0)^W\] is an isomorphism for all $M_0$. From this, we immediately deduce that the unit map is an isomorphism, and so the left adjoint $\Symt \otimes_{\Symt^W} -$ is fully faithful. 

We now characterize the essential image: observe that if $c(M)$ is an isomorphism, then $M$ certainly lies in the essential image of $\Symt \otimes_{\Symt^W} -$. Conversely, if $M \cong \Symt \otimes_{\Symt^W} M_0$ for some $M_0 \in \Symt^W\mmod$, we observe that the composite of \[\Symt \otimes_{\Symt^W} M_0\xrightarrow{p \otimes m_0 \mapsto p \otimes 1 \otimes m_0} \Symt \otimes_{\Symt^W} (\Symt \otimes_{\Symt^W} M_0)^W\] (i.e. the map $L(u(M_0))$) with the map \[\Symt \otimes_{\Symt^W} (\Symt \otimes_{\Symt^W} M_0)^W \xrightarrow{p \otimes p' \otimes m_0 \mapsto pp' \otimes m_0} \Symt \otimes_{\Symt^W} M_0\] is the identity. Since $L(u(M_0))$ is an isomorphism (as $u(M_0)$ is an equivalence), we deduce that $c(L(M_0)) \cong c(M)$ is an isomorphism, as desired.
\end{proof}


\begin{Proposition}\label{Forgetful Functor form nilHeckeAlgebra to S rtimes W mod fully faithful}
The forgetful functor $\nilHeckeAlgebra\mmod \to \Symt\mmod^W$ is fully faithful. Moreover: \begin{enumerate}
    \item The essential image of this functor consists of those $M \in \Symt\mmod^W$ for which the module map $\Symt \otimes_{\Symt^W} M^W \to M$ is an isomorphism. 
    \item For any $M' \in \nilHeckeAlgebra\mmod$, the action map $\nilHeckeAlgebra \otimes_{\Symt \rtimes W} M' \to M'$ is an isomorphism.
\end{enumerate}
\end{Proposition}

\begin{proof}
The forgetful functor has a left adjoint $\nilHeckeAlgebra \otimes_{\Symt \rtimes W} -$, and, as in the proof of Proposition \ref{Fully faithfulness of tensoring up even for algebra objects}, the counit $c(M)$ for a fixed $M \in \nilHeckeAlgebra\mmod$ is given by the formula $h \otimes m \mapsto hm$. To show that the forgetful functor is fully faithful, it suffices to show that the counit is an isomorphism for every module $M$ over the nil-Hecke algebra $\nilHeckeAlgebra$. (Of course, this will also show (2).) Since \begin{equation}\xymatrix@R+2em@C+2em{\nilHeckeAlgebra \otimes_{\Symt \rtimes W} \nilHeckeAlgebra \otimes_{\nilHeckeAlgebra} M \ar[d]^{\mathrm{id} \otimes \mathrm{mult}}\ar[r]^{\textcolor{white}{whitee}c(\nilHeckeAlgebra) \otimes \mathrm{id}_M} & \nilHeckeAlgebra \otimes_{\nilHeckeAlgebra} M \ar[d]^{\mathrm{mult}}\\
\nilHeckeAlgebra \otimes_{\Symt \rtimes W} M \ar[r]^{c(M)}&  M
  }\end{equation} commutes and the vertical arrows are isomorphisms, it suffices to prove that $c(\nilHeckeAlgebra)$ itself is an isomorphism. Since $c(\nilHeckeAlgebra)(1 \otimes h) = h$ for all $h \in \nilHeckeAlgebra$, the map $c(\nilHeckeAlgebra)$ is surjective. For injectivity, observe that, for any $\sum_i h_i \otimes h_i' \in \nilHeckeAlgebra \otimes_{\Symt \rtimes W} \nilHeckeAlgebra$ such that $\sum_i h_i h_i' = 0$ in $\nilHeckeAlgebra$, there exists some nonzero $p \in \Symt$ such that $ph_i \in \Symt \rtimes W$ for all $i$. Therefore \[p\sum_i h_i \otimes h_i' = \sum_i ph_i \otimes h_i' = \sum_i 1 \otimes ph_ih_i' = \sum_i p \otimes h_ih_i'= 0\] in $\nilHeckeAlgebra \otimes_{\Symt \rtimes W} \nilHeckeAlgebra$. Next, we observe that $\nilHeckeAlgebra \otimes_{\Symt \rtimes W} \nilHeckeAlgebra$ is torsion-free as a left $\Symt$-module: indeed, $\nilHeckeAlgebra \otimes_{\Symt \rtimes W} \nilHeckeAlgebra$ is isomorphic to the $W$-invariants of the free $\Symt$-module $\nilHeckeAlgebra \otimes_{\Symt} \nilHeckeAlgebra$. Since $p\sum_i h_i \otimes h_i' = 0$, we deduce that $\sum_i h_i \otimes h_i' = 0$, as required. The description of the essential image is known, see for example \cite[Lemma 7.1.5]{Gin}.
\end{proof}
The following Lemma, which follows by combining \cite[Theorem 1.1]{LoRemark} and \cite[Lemma 7.1.5]{Gin}, will also be used in what follows: 

\begin{Lemma}\label{Closed Under Demazure Implies Recoverable from GIT Quotient}
Let $M$ be a  torsion-free $\Symt$-module equipped with a compatible $W$-action. If $M$ is closed under the action of $D_s$ for every simple reflection $s$, then the multiplication map \[\Symt \otimes_{\Symt^W} M^W \mathrel{\to} M\] is an isomorphism. 
\end{Lemma}

The following fact is an immediate consequence of Proposition \ref{standard w0}(2). 

\begin{Proposition}\label{Demazure envelope is naive Demazure envelope for modules}
  For $M$ as in Definition \ref{Demazure envelope of module definition}, there is an
  equality 
  \[E_{\mathrm{mod}}(M)=\left\{\sum_wD_w(m_w) \;:\; w \in W,\ m_w \in M\right\}.
    \]
\end{Proposition}

\begin{Corollary}\label{Trivial Part of EmodM is Dw0 of M}
For $M$ as in Definition \ref{Demazure envelope of module definition}, we have $\Emod(M)^W = D_{w_0}(\Emod(M)) = D_{w_0}(M)$. 
\end{Corollary}

\begin{proof}
Since $D_sD_{w_0} = 0$ for any simple reflection $s$, we deduce that $D_{w_0}(\Emod(M)) \subseteq \Emod(M)^W$. Conversely, if $m \in \Emod(M)^W$, then $\Delta m \in \Emod(M)$ satisfies $D_{w_0}(\Delta m) = |W|m$ by Proposition \ref{standard w0}(1). We deduce that $D_{w_0}\left(\frac{1}{|W|}\Delta m\right) = m$ and so $D_{w_0}(\Emod(M)) \supseteq \Emod(M)^W$. 

Next, we prove $D_{w_0}(\Emod(M)) = D_{w_0}(M)$. Clearly $D_{w_0}(\Emod(M)) \supseteq D_{w_0}(M)$. We now prove $D_{w_0}(\Emod(M)) \subseteq D_{w_0}(M)$. Choose some $D_{w_0}(m') \in D_{w_0}(\Emod(M))$. By Proposition \ref{Demazure envelope is naive Demazure envelope for modules}, we may write $m' = \sum_wD_w(m_w)$ for some $m_w \in M$. Since $D_{w_0}D_s = 0$ for all simple reflections $s$, we deduce that \[D_{w_0}(m') = D_{w_0}(\sum_wD_w(m_w)) = \sum_wD_{w_0}D_w(m_w) = D_{w_0}(m_1)\] and so $D_{w_0}(m') \in D_{w_0}(M)$ as desired.
\end{proof}

We also give another description of the Demazure envelope:

\begin{Proposition}\label{Demazure Envelope is Image of NilHecke Action Map}
For $M$ as in Definition \ref{Demazure envelope of module definition}, the action map induces an isomorphism \[\nilHeckeAlgebra \otimes_{\Symt \rtimes W} M \xrightarrow{\sim} \Emod(M).\] 
\end{Proposition}

\begin{proof}
By Proposition \ref{Demazure envelope is naive Demazure envelope for modules}, this map is surjective. We now verify injectivity. Note that the action map factors as a composite \[\nilHeckeAlgebra \otimes_{\Symt \rtimes W} M \xrightarrow{\mathrm{id}_{\nilHeckeAlgebra} \otimes \imath} \nilHeckeAlgebra \otimes_{\Symt \rtimes W} \Emod(M) \xrightarrow{\mathrm{action}} \Emod(M),\] where $\imath$ is the inclusion map, and the rightmost arrow in this composite is an isomorphism by Proposition \ref{Forgetful Functor form nilHeckeAlgebra to S rtimes W mod fully faithful}(2). Therefore it suffices to prove that $\mathrm{id}_{\nilHeckeAlgebra} \otimes \imath$ is injective, for which it is enough to know that $\nilHeckeAlgebra \otimes_{\Symt \rtimes W} -$ is exact. But we can factor this functor as a composite of $\nilHeckeAlgebra \otimes_{\Symt} -$ and $W$-invariants, the former of which is exact since  $\nilHeckeAlgebra$ is a free $\Symt$-module with basis $D_w,\, w \in W$, see \cite[Corollary 1]{DemazureInvariantsSymmetriquesEntiersDesGroupsDeWeylEtTorsion}, and the latter is exact as well.
\end{proof}

\begin{proof}[Proof of Proposition \ref{env prop}(i)]
Let $E'(A)$ denote the subalgebra of $\Frac(\t)\otimes_{\Symt} A$ generated by the elements $D_w(a)$ for $a \in A$ and $w \in W$, and let $E''(A)$ denote the smallest subalgebra of $\Frac(\t)\otimes_{\Symt} A$ containing $A$ and the elements $D_{w_0}(a)$ for $a \in A$. We prove that $E(A) = E'(A) = E''(A)$. We certainly have $E'(A) \subseteq E(A)$. To prove that $E(A) \subseteq E'(A)$, it suffices to prove that $E'(A)$ is closed under the Demazure operators. To this end, we first claim that $E'(A)$ is $W$-stable. To prove this claim, note that, since each $v \in W$ acts by algebra automorphisms, it is enough to show that $v(D_w(a))$ lies in $E'(A)$ for all $v, w \in W$ and $a \in A$. Now, by Proposition \ref{standard w0}(2), we can write $vD_wv^{-1} = \sum_{w' \in W}D_{w'}p_{w'}$ for some $p_{w'} \in \Symt$, so $vD_w= \sum_{w' \in W}D_{w'}p_{w'}v$. Therefore \[v(D_w(a)) = \sum_{w'}D_{w'}(p_{w'}va)\] lies in $E'(A)$ since $A$ itself is a $W$-stable $\Symt$-algebra.

Now, since each Demazure operator $D_{s_{\alpha}}$ is $k$-linear, it suffices to prove that \[D_{s_{\alpha}}(D_{w_1}(a_1)\mathinner{\cdots}D_{w_n}(a_n)) \in E'(A)\] for all $w_i \in W$ and $a_i \in A$. By the twisted Leibniz rule for Demazure operators (\ref{Twisted Leibniz rule}) we have \[\begin{aligned}
D_{s_{\alpha}}(D_{w_1}(a_1)\cdots D_{w_n}(a_n))
&= D_{s_{\alpha}}(D_{w_1}(a_1))D_{w_2}(a_2)\cdots D_{w_n}(a_n) \\
&\quad + s_{\alpha}(D_{w_1}(a_1))D_{s_{\alpha}}(D_{w_2}(a_2)\cdots D_{w_n}(a_n))
\end{aligned}\] for any simple reflection $s_{\alpha} \in W$. The first term in this sum is $0$ or $D_{s_{\alpha}w_1}(a_1)D_{w_2}(a_2)\mathinner{\cdots}D_{w_n}(a_n)$, and thus lies in $E'(A)$. By the $W$-stability of $E'(A)$, we see that \[s_{\alpha}(D_{w_1}(a_1))D_{s_{\alpha}}(D_{w_2}(a_2)\mathinner{\cdots}D_{w_n}(a_n)) \in E'(A) \text{ if } D_{s_{\alpha}}(D_{w_2}(a_2)\mathinner{\cdots}D_{w_n}(a_n)) \in E'(A)\] and the claim follows by induction. Therefore $E(A) = E'(A)$.

We now prove that $E''(A) = E'(A)$. We certainly have $E''(A) \subseteq E'(A)$. As above, we let $E_{\mathrm{mod}}(A)$ denote the smallest $\Symt$-sub\textit{module} of $\Frac(\t)\otimes_{\Symt} A$ containing $A$ and closed under the Demazure operators; observe that $E_{\mathrm{mod}}(A)$ generates $E'(A)$ as an algebra. By Corollary \ref{Trivial Part of EmodM is Dw0 of M}, we have $E_{\mathrm{mod}}(A)^W = D_{w_0}(A) \subseteq E''(A)$. Since the multiplication map $\Symt \otimes_{\Symt^W} E_{\mathrm{mod}}(A)^W \to E_{\mathrm{mod}}(A)$ is an isomorphism by Lemma \ref{Closed Under Demazure Implies Recoverable from GIT Quotient}, every element of $E_{\mathrm{mod}}(A)$ is a sum of elements of the form $pe'$ with $p \in \Symt$ and $e' \in E_{\mathrm{mod}}(A)^W$. As $E''(A)$ is a $\Symt$-algebra, this implies $E_{\mathrm{mod}}(A) \subseteq E''(A)$. Since $E_{\mathrm{mod}}(A)$ generates $E'(A)$ as an algebra, we conclude that $E'(A) \subseteq E''(A)$, as desired. 
\end{proof}

\begin{Remark}
    There are examples of algebras $A$ as in Definition \ref{Demazure envelope of algebra definition} for which the smallest subalgebra containing $A$ and the elements $D_{s}(a)$, for $a \in A$ and all simple reflections $s$, is strictly smaller than the Demazure envelope $E(A)$.
\end{Remark}

\section{\texorpdfstring{Weil restriction and the Demazure envelope}{Weil restriction and the Demazure envelope}}\label{Invariant Weil Restrictions Subsubsection} \newcommand{\LTdGeneralizationforWeilRestriction}{\mathbf{t}}
\newcommand{\LCGeneralizationForWeilRestriction}{\mathbf{c}}
\subsection{\texorpdfstring{Weil restriction}{Weil restriction}}Assume we are given a finite locally free morphism of affine schemes $p: \LTdGeneralizationforWeilRestriction \to \LCGeneralizationForWeilRestriction$. Recall that, for any affine scheme $X$ over $\LTdGeneralizationforWeilRestriction$, there exist a scheme $\Res_{p}(X)$ and a map $c(X): \Res_{p}(X)_{\LTdGeneralizationforWeilRestriction} \to X$ known as the \textit{counit map} such that, for any affine scheme $S$ over $\LCGeneralizationForWeilRestriction$, the composite of the maps \begin{equation}\label{General adjunction for Weil restriction Generalization}\Maps_{\LCGeneralizationForWeilRestriction}(S, \Res_{p}(X)) \xrightarrow{\LTdGeneralizationforWeilRestriction \times_{\LCGeneralizationForWeilRestriction} -} \Maps_{\LTdGeneralizationforWeilRestriction}(S_{\LTdGeneralizationforWeilRestriction}, \Res_{p}(X)_{\LTdGeneralizationforWeilRestriction}) \xrightarrow{c(X) \circ -} \Maps_{\LTdGeneralizationforWeilRestriction}(S_{\LTdGeneralizationforWeilRestriction}, X) \end{equation} is a bijection, see for example \cite[Section 7.6]{BoschLutkebohmertRaynaudNeronModels}. From this, one can check that the Weil restriction, together with its counit map, is unique up to unique isomorphism. We also recall (see \cite[Section 7.6]{BoschLutkebohmertRaynaudNeronModels}) that $\Res_{p}(X)$ is smooth over $\LCGeneralizationForWeilRestriction$ if $X$ is smooth over $\LTdGeneralizationforWeilRestriction$. When the map $p$ is clear from context, we will omit it from the notation and use the shorthand $\Res(X)$ for $\Res_{p}(X)$.

Let us now assume that a finite group $\Gamma$ acts on $\LTdGeneralizationforWeilRestriction$ in such a way that $p$ is $\Gamma$-equivariant for the trivial action on $\LCGeneralizationForWeilRestriction$, and that $X$ admits a compatible $\Gamma$-action. In this case, one can define a $\Gamma$-action on $\Res(X)$ for which the counit map is $\Gamma$-equivariant: indeed, for every $\gamma \in \Gamma$, we let $\overline{\gamma}$ denote the preimage of the map $(r, t) \mapsto \gamma \cdot c(X)(r, \gamma^{-1}t)$ in $\Maps_{\LTdGeneralizationforWeilRestriction}(\Res(X)_{\LTdGeneralizationforWeilRestriction}, X)$ under the bijection given by the composite of the maps in (\ref{General adjunction for Weil restriction Generalization}). If $S$ is a scheme over $\LCGeneralizationForWeilRestriction$ equipped with a compatible action of $\Gamma$, one can check that the maps in (\ref{General adjunction for Weil restriction Generalization}) are $\Gamma$-equivariant. We thereby obtain bijections \begin{equation}\label{General adjunction for fixed points of Weil restriction generalization} \Maps_{\LCGeneralizationForWeilRestriction}^\Gamma(S, \Res(X)) \xrightarrow{\sim} \Maps_{\LTdGeneralizationforWeilRestriction}^\Gamma(\LTdGeneralizationforWeilRestriction \times_{\LCGeneralizationForWeilRestriction} S, X)\end{equation} given by the composite of the maps in (\ref{General adjunction for Weil restriction Generalization}). In particular, if $S$ has the trivial $\Gamma$-action, we obtain a bijection \[\Maps_{\LCGeneralizationForWeilRestriction}(S, \Res(X)^\Gamma) \xrightarrow{\sim} \Maps_{\LTdGeneralizationforWeilRestriction}^\Gamma(\LTdGeneralizationforWeilRestriction \times_{\LCGeneralizationForWeilRestriction} S, X).\]

In what follows, we will use the following elementary observation, which is proved in \cite[Section 7.6]{BoschLutkebohmertRaynaudNeronModels} and is not difficult to verify directly; see also \cite[Lemma 2.3]{HekkingKhanRydhDeformationtotheNormalBundleandBlowUpsViaDerivedWeilRestrictions} for the derived variant:

\begin{Proposition}\label{Base Change of Weil Restriction}
Let $\LCGeneralizationForWeilRestriction' \to \LCGeneralizationForWeilRestriction$ be a map of affine schemes, set $\LTdGeneralizationforWeilRestriction' := \LTdGeneralizationforWeilRestriction \times_{\LCGeneralizationForWeilRestriction} \LCGeneralizationForWeilRestriction'$, and denote by $p': \LTdGeneralizationforWeilRestriction' \to \LCGeneralizationForWeilRestriction'$ the projection map. Then the counit map \[(\Res_{p}(X) \times_{\LCGeneralizationForWeilRestriction} \LCGeneralizationForWeilRestriction') \times_{\LCGeneralizationForWeilRestriction'} \LTdGeneralizationforWeilRestriction' \cong (\Res_{p}(X) \times_{\LCGeneralizationForWeilRestriction} \LTdGeneralizationforWeilRestriction) \times_{\LTdGeneralizationforWeilRestriction } \LTdGeneralizationforWeilRestriction' \xrightarrow{c(X) \times \mathrm{id}_{\LTdGeneralizationforWeilRestriction'}} X_{\LTdGeneralizationforWeilRestriction'}\] induces a canonical isomorphism \[\Res_{p}(X) \times_{\LCGeneralizationForWeilRestriction} \LCGeneralizationForWeilRestriction' \cong \Res_{p'}(X \times_{\LTdGeneralizationforWeilRestriction} \LTdGeneralizationforWeilRestriction')\] of schemes over $\LCGeneralizationForWeilRestriction'$. If $\LTdGeneralizationforWeilRestriction, \LCGeneralizationForWeilRestriction, \LCGeneralizationForWeilRestriction'$, and $X$ have $\Gamma$-actions for which the maps $\LCGeneralizationForWeilRestriction' \to \LCGeneralizationForWeilRestriction$, $p$, and $X \to \LTdGeneralizationforWeilRestriction$ are $\Gamma$-equivariant, then this isomorphism is compatible with the $\Gamma$-action.
\end{Proposition}
We will be interested in the special case in which $\LTdGeneralizationforWeilRestriction = \LTd$ and $\LCGeneralizationForWeilRestriction = \LC$. In this case, the counit map of the invariant Weil restriction is generically an isomorphism on $\LTd$:

\begin{Proposition}\label{Counit Map for Invariant Weil Restriction is Generically an Isomorphism}
Let $\LC_{\mathrm{reg}}$ denote the image in $\LC$ of the complement $\LT^*_{\mathrm{reg}}$ of the coroot hyperplanes in $\LTd$. Then, for any affine scheme $X$ over $\LTd$ equipped with a compatible $W$-action, the counit map $c(X): \Res^W(X)_{\LTd} \to X$ is an isomorphism over $\LT^*_{\mathrm{reg}}$.
\end{Proposition}

\begin{proof}
For each $w \in W$, let $\Gamma_w$ denote the graph of the action of $w$, i.e. the closed embedding $(\mathrm{act}_w, \mathrm{id}_{\LTd}): \LTd \to \LTd \times \LTd$. Denote by $\Gamma_W$ the union of these graphs, or in other words the closed subscheme of $\LTd \times \LTd$ cut out by the intersection of the ideals of the $\Gamma_w$. We observe that the squares, and thus the rectangle, of the diagram \begin{equation*}\xymatrix@R+2em@C+2em{W \times \LT^*_{\mathrm{reg}} \textcolor{white}{w} \ar[r]^{\textcolor{white}{w}(w, \varphi) \mapsto (w\varphi, \varphi)} \ar[d]^{(w, \varphi) \mapsto w\varphi} & \Gamma_W \ar[r]^{(wx, x) \mapsto x} \ar[d]^{(wx, x) \mapsto wx}  & \LTd \ar[d]\\
\LT^*_{\mathrm{reg}}  \ar[r]^{\subseteq} & \LTd \ar[r] & \LC 
  }\end{equation*} are Cartesian: indeed, the fact that the rightmost square is Cartesian is standard (see for example \cite[Proposition A.2]{GannonDescentToTheCoarseQuotientForPseudoreflectionAndAffineWeylGroups}), and the claim for the leftmost square follows from the fact that $W$ acts freely on $\LT^*_{\mathrm{reg}}$. 

By comparing the counit maps of Proposition \ref{Base Change of Weil Restriction}, we deduce that it suffices to show that \[\Res^W_{W \times \LT^*_{\mathrm{reg}} \xrightarrow{\mathrm{act}}\LT^*_{\mathrm{reg}}}(X_{W \times \LT^*_{\mathrm{reg}}})_{W \times \LT^*_{\mathrm{reg}}} \to X_{W \times \LT^*_{\mathrm{reg}}}\] is an isomorphism, where, as above, we use the subscript $Y_{W \times \LT^*_{\mathrm{reg}}}$ to denote the base change of a scheme $Y$ over $\LTd$ to $W \times \LT^*_{\mathrm{reg}}$. However, it is not difficult to check that, for \textit{any} finite flat morphism of affine schemes $p: \LTdGeneralizationforWeilRestriction \to \LCGeneralizationForWeilRestriction$ for which the $W$-action on $\LTdGeneralizationforWeilRestriction$ is free and $\LCGeneralizationForWeilRestriction = \LTdGeneralizationforWeilRestriction\sslash W$, the counit map $\Res_p^W(X)_{\LTdGeneralizationforWeilRestriction} \to X$ is an isomorphism for any $W$-scheme $X$ over $\LTdGeneralizationforWeilRestriction$. Indeed, by passing to a cover, we may assume that $p$ has the form $\LTdGeneralizationforWeilRestriction = W \times \LCGeneralizationForWeilRestriction \to \LCGeneralizationForWeilRestriction$. In this case, any $X$ over $\LTdGeneralizationforWeilRestriction$ with a compatible $W$-action has the form $W \times X_0$ for some scheme $X_0$ over $\LCGeneralizationForWeilRestriction$. One can directly check that the map \[X_0 \times_{\LCGeneralizationForWeilRestriction} \LTdGeneralizationforWeilRestriction \cong W \times X_0 \cong X\] has the desired universal property, so that $\Res^W(X) \cong X_0$.
\end{proof} 
\subsection{Proof of Theorem \ref{intro thm}}Hereafter, we assume that $X = \Spec(A)$ is an affine $W$-scheme, with $A$ an integral domain, equipped with a $W$-equivariant dominant map to $\LTd$. 

\subsubsection*{Universal property of the Demazure envelope}We first show that the Demazure envelope has the same universal property as $\Res^W(\Spec(A))$ when we restrict to torsion-free $\Symt^W$-modules. More precisely, we consider the map \begin{equation*}\label{Pull back from Demazure envelope map}\Maps_{\mathrm{CAlg}(\Symt\mmod)}^W(E(A), \Symt \otimes_{\Symt^W} B) \to \Maps_{\mathrm{CAlg}(\Symt\mmod^W)}(A, \Symt \otimes_{\Symt^W} B)\end{equation*} obtained by pulling back along the inclusion $A \mathrel{\hookrightarrow} E(A)$, where $\mathrm{CAlg}(\Symt\mmod)$ denotes the category of commutative $\Symt$-algebras. In the language of affine schemes, we are therefore considering the map \begin{equation}\label{Pull back from Demazure envelope map for affine schemes}\Maps^W_{\LTd}(\LTd \times_{\LC} \Spec(B), \Spec(E(A))) \to \Maps^W_{\LTd}(\LTd \times_{\LC} \Spec(B), \Spec(A)) \end{equation} obtained by postcomposition with the map $\Spec(E(A)) \to \Spec(A)$. 
\begin{Lemma}\label{Demazure Envelope Has Adjunction Property for Torsion Free Schemes Algebra Version}
Assume $A$ is as in Definition \ref{Demazure envelope of algebra definition}, and let $B$ be a $\Symt^W$-algebra. The map (\ref{Pull back from Demazure envelope map for affine schemes}) is injective. Moreover, this map is surjective if $B$ is torsion-free as a $\Symt^W$-module. 
\end{Lemma}

\begin{proof}
    We first verify injectivity: assume $f, g: E(A) \to \Symt \otimes_{\Symt^W} B$ are $\Symt$-algebra morphisms compatible with the $W$-action such that the restrictions to $A$ agree. Since the forgetful functor from modules for the nil-Hecke ring to modules for $\Symt \rtimes W$ is fully faithful (Proposition \ref{Forgetful Functor form nilHeckeAlgebra to S rtimes W mod fully faithful}), we deduce that $f$ and $g$ are compatible with the Demazure operators. By Proposition \ref{env prop}(i), $E(A)$ is generated as an algebra by the elements $D_w(a)$ for $a \in A$ and $w \in W$. Since \[f(D_w(a)) = D_w(f(a)) = D_w(g(a)) = g(D_w(a))\] for all such $a$ and $w$, we deduce that $f=g$, as desired.

    Now, assuming that $B$ is torsion-free as a $\Symt^W$-module, we verify surjectivity. Assume $h: A \to \Symt \otimes_{\Symt^W} B$ is a $W$-equivariant map of $\Symt$-algebras. By Proposition \ref{Demazure envelope is naive Demazure envelope for modules}, every element of $E_{\mathrm{mod}}(A)$ has the form $\sum_wD_w(a_w)$ for some $a_w \in A$. Define the map $\tilde{h}: E_{\mathrm{mod}}(A) \to \Symt \otimes_{\Symt^W} B$ by the formula \[\tilde{h}\big(\sum_wD_w(a_w)\big) := \sum_wD_w(h(a_w)).\] Choose $n$ such that $\Delta^nD_w \in \Symt \rtimes W$ for all $w \in W$. If $\sum_wD_w(a_w) = \sum_wD_w(a_w')$, then \begin{equation}\label{Pi to n times Maps of Demazures Agree}\Delta^n\sum_wD_w(h(a_w)) = h\big(\sum_w\Delta^nD_w(a_w)\big) = h\big(\sum_w\Delta^nD_w(a_w')\big) = \Delta^n\sum_wD_w(h(a_w')).\end{equation} As $B$ is $\Symt^W$-torsion-free, $\Symt \otimes_{\Symt^W} B$ is $\Symt$-torsion-free, and so from (\ref{Pi to n times Maps of Demazures Agree}) we deduce that $\sum_wD_w(h(a_w)) = \sum_wD_w(h(a_w'))$ and thus $\tilde{h}$ is well-defined. 
    
     Let $Q$ denote the field of fractions of $\Symt$. Then $h$ has an extension to a $W$-equivariant algebra map \[f: Q\otimes_{\Symt} A \to
 Q\otimes_{\Symt}\Symt\otimes_{\Symt^W} B  \cong  Q\otimes_{\Symt^W} B.\] 
 Since $B$ is torsion-free as a $\Symt^W$-module, the map
$\Symt \otimes_{\Symt^W} B \to Q\otimes_{\Symt^W} B$ is injective 
and we may view $\Symt \otimes_{\Symt^W} B$ as a
 subalgebra of $Q\otimes_{\Symt^W} B$.
 Since $E(A)$ is a subalgebra of $Q\otimes_{\Symt} A$, it suffices to show that $f(E(A))$ is contained in the subalgebra
$\Symt \otimes_{\Symt^W} B$. As $f$ is $W$-equivariant, it commutes with the Demazure operators. Therefore $f(D_w(a))=D_w(h(a))=\tilde{h}(D_w(a))$ for all $a \in A$ and $w \in W$. Our claim now follows since every element of $E(A)$ is a sum of terms of the form $D_{w_1}(a_1)\mathinner{\cdots}D_{w_n}(a_n)$ and \[f(D_{w_1}(a_1)\mathinner{\cdots}D_{w_n}(a_n)) = f(D_{w_1}(a_1))\mathinner{\cdots}f(D_{w_n}(a_n)) = \tilde{h}(D_{w_1}(a_1))\mathinner{\cdots}\tilde{h}(D_{w_n}(a_n))\] lies in $\Symt \otimes_{\Symt^W}B$ by our previous analysis. 
\end{proof}

\subsubsection*{Construction of $\Res^W_{\circ}(X)$}\label{Construction of Component}Observe that, by Proposition \ref{Counit Map for Invariant Weil Restriction is Generically an Isomorphism}, the counit map \[\Res^W(X)_{\LTd} = \LTd \times_{\LC} \Res^W(X)\to X\] is an isomorphism away from the coroot hyperplanes.  In particular, $\Res^W(X)$ contains an open subset isomorphic to $\Spec((A\left[\frac{1}{\Delta^2}\right])^W)$, which is reduced and irreducible if $A$ is an integral domain. We let $\ResFreeClosureW(X)$ denote the scheme-theoretic closure of this open subset. 

We now construct a canonical map \begin{equation}\label{Map from universal property} \Spec(E(A)^W) \to \Res^W(X)\end{equation} as follows. Since we have an isomorphism $e: \Spec(E(A)) \xrightarrow{\sim} \LTd \times_{\LC} \Spec(E(A)^W)$ by Lemma \ref{Closed Under Demazure Implies Recoverable from GIT Quotient}, the universal property of Weil restriction gives a map $f: \Spec(E(A)^W) \to \Res^W(X)$ over $\LC$ whose base change to $\LTd$ commutes with the counit map $c(X)$ in the natural way. Using this map, we may state the following proposition, which, combined with Proposition \ref{env prop}, proves the first assertion in Theorem \ref{intro thm}. 

\begin{Proposition}\label{ResFreeClosureW is Spec of W Invariants of Demazure Operator}
Let $X = \Spec(A)$ be a reduced and irreducible affine $W$-scheme equipped with a compatible map to $\LTd$. The map (\ref{Map from universal property}) induces an isomorphism $\Spec(E(A)^W) \xrightarrow{\sim} \ResFreeClosureW(X)$. 
\end{Proposition}

\begin{proof}
We claim that if $\mathcal{Y}$ denotes either of the schemes $\Spec(E(A)^W)$ or $\ResFreeClosureW(X)$, then $\mathcal{Y}$, together with the map $\mathcal{Y}_{\LTd} \to X$, satisfies the following universal property: for any $\Symt^W$-algebra $B$ which is torsion-free as a $\Symt^W$-module and any $W$-equivariant map $\varepsilon: \LTd \times_{\LC} \Spec(B) \to X$ of affine schemes over $\LTd$, there exists a unique map $\varphi: \Spec(B) \to \mathcal{Y}$ of schemes over $\LC$ such that $c \circ \varphi_{\LTd} = \varepsilon$, where $c$ is the counit map when $\mathcal{Y} = \ResFreeClosureW(X)$, and is the composite $\LTd \times_{\LC} \Spec(E(A)^W) \cong \Spec(E(A)) \to \Spec(A)$ when $\mathcal{Y} = \Spec(E(A)^W)$. We have already shown this for $\mathcal{Y} = \Spec(E(A)^W)$ in Lemma \ref{Demazure Envelope Has Adjunction Property for Torsion Free Schemes Algebra Version}, so we now treat the case $\mathcal{Y} = \ResFreeClosureW(X)$. 

Assume we have a $\Symt^W$-algebra $B$ which is torsion-free as a $\Symt^W$-module and a $W$-equivariant map $\varepsilon: \LTd \times_{\LC} \Spec(B) \to X$ of affine schemes over $\LTd$. Then, by the usual universal property of $\Res^W(X)$, there exists a unique map $\varphi: \Spec(B) \to \Res^W(X)$ which is compatible with the counit map in the natural way; it suffices to prove that $\varphi$ factors through $\ResFreeClosureW(X)$. To this end, let $C := \O(\Res^W(X))$ and $I := \Ker(C \to C\left[\frac{1}{\Delta^2}\right])$, so that $\ResFreeClosureW(X) = \Spec(C/I)$. For any $r \in I$, there exists some $n$ such that $\Delta^{2n}r = 0$. Therefore $\Delta^{2n}\varphi^*(r) = 0$ in $B$, and, since $B$ is torsion-free as a $\Symt^W$-module, we deduce that $\varphi^*(r) = 0$. Thus $\varphi$ factors through $\ResFreeClosureW(X)$, as desired. %

Both $\Spec(E(A)^W)$ and $\ResFreeClosureW(X)$ are affine schemes over $\LC$ whose rings of functions are torsion-free over $\O(\LC)$, and we have now shown that both satisfy the above universal property. Our claim therefore follows, since an object satisfying a universal property is unique up to unique isomorphism.
\end{proof}

\subsubsection*{Proof of Theorem \ref{intro thm}} In Proposition \ref{ResFreeClosureW is Spec of W Invariants of Demazure Operator}, we proved the first assertion of Theorem \ref{intro thm}. We now prove parts (i) and (ii). We have already recalled that, if $X \to \LTd$ is smooth, then $\Res(X)$ is smooth. Following \cite[Lemme 2.4.1]{Ngo}, one then uses the fact that the fixed points of a finite group acting on a smooth scheme are smooth and computes the relevant differential to show that $\Res^W(X) \to \LC$ is smooth. In particular, the ring of functions on $\Res^W(X)$ is torsion-free as an $\O(\LC)$-module. Thus, Theorem \ref{intro thm}(ii) follows from Theorem \ref{intro thm}(i). In turn, Theorem \ref{intro thm}(i) follows as a special case of the following Theorem, which is one of the main results of this paper:

\begin{Theorem}\label{Comparison of ResW and Demazure Envelope}Let $A$ be a finitely generated integral $k$-algebra equipped with an injective map $\Symt \to A$ and a compatible $W$-action. The following are equivalent: \begin{enumerate}
    \item[(A)] The canonical map $\Spec(E(A)^W) \to \Res^W(\Spec(A))$ is an isomorphism. 
    \item[(A')] The base change of the canonical map \[\Spec(E(A)) \xrightarrow{\sim} \LTd\times_{\LC} \Spec(E(A)^W) \to \LTd\times_{\LC}\Res^W(\Spec(A))\] is an isomorphism.
    \item[(B)] The affine scheme $\Res^W(\Spec(A))$ is integral (i.e. its ring of functions is an integral domain) of dimension $\dim(A)$, and the map of rings $\Symt^W \to \O(\Res^W(\Spec(A)))$ is injective.
    \item[(B')] The affine scheme $\Res^W(\Spec(A))_{\LTd}$ is integral of dimension $\dim(A)$ and the ring map $\Symt \to \O(\Res^W(\Spec(A))_{\LTd})$ is injective.
    \item[(C)] The ring of functions on $\Res^W(\Spec(A))$ is torsion-free as a $\Symt^W$-module.
    \item[(C')] The ring of functions on $\Res^W(\Spec(A))_{\LTd}$ is torsion-free as a $\Symt$-module.
    \item[(D)] The map (\ref{Pull back from Demazure envelope map for affine schemes}) is surjective for any $\Symt^W$-algebra $B$.
\end{enumerate}
\end{Theorem}

\begin{proof}
    Certainly (A) $\implies$ (A'). Conversely, since taking $W$-invariants of the corresponding rings of functions is exact, (A') $\implies$ (A). We now show (A) implies (B); assume (A). Since $\Symt \to A$ is injective, so is $\Symt^W \to A^W$, and $E(A)^W$ is an integral domain since it is the ring of $W$-invariants of a subring of the field of fractions of the integral domain $A$. Finally, since $\Res^W(\Spec(A))$ is integral, we may compute its dimension on any nonempty open subset. Therefore the dimension of $\Res^W(\Spec(A))$  equals the dimension of $\Spec(A)$ by Proposition \ref{Counit Map for Invariant Weil Restriction is Generically an Isomorphism}. Similarly, (A') implies (B'). 

    To see that (B) implies (C), suppose that $\O(\Res^W(\Spec(A)))$ is not torsion-free as a $\Symt^W$-module. Then there exists some nonzero $f \in \O(\Res^W(\Spec(A)))$ and some nonzero $x \in \Symt^W$ such that $\overline{x}f = 0$, where $\overline{x}$ is the image of $x$ in $\O(\Res^W(\Spec(A)))$. If $\overline{x} = 0$, then the map $\Symt^W \to \O(\Res^W(\Spec(A)))$ is not injective; if $\overline{x}$ is nonzero, then $\O(\Res^W(\Spec(A)))$ is not an integral domain. Similarly, (B') implies (C'). 

    Now assume that (C') holds. Then, by the translation of Lemma \ref{Demazure Envelope Has Adjunction Property for Torsion Free Schemes Algebra Version} to affine schemes, there exists a unique map $\varphi: \Res^W(\Spec(A)) \to \Spec(E(A)^W)$ such that its base change $\varphi_{\LTd}$ makes the diagram \begin{equation}\xymatrix@R+2em@C+2em{\Res^W(\Spec(A))_{\LTd} \ar[d]^{c(\Spec(A))} \ar[r]^{\varphi_{\LTd}} & \Spec(E(A)^W)_{\LTd}\\
\Spec(A) & \Spec(E(A)) \ar[l] \ar[u]_{\sim}
  }\end{equation} commute, where the upward-pointing arrow is the isomorphism of Lemma \ref{Closed Under Demazure Implies Recoverable from GIT Quotient}. However, the universal property of $\Res^W(\Spec(A))$ gives a map $\varphi': \Spec(E(A)^W) \to \Res^W(\Spec(A))$ making the diagram \begin{equation}\xymatrix@R+2em@C+2em{\Res^W(\Spec(A))_{\LTd} \ar[d]^{c(\Spec(A))}  & \Spec(E(A)^W)_{\LTd} \ar[l]_{\varphi'_{\LTd}}\\
\Spec(A) & \Spec(E(A)) \ar[l] \ar[u]_{\sim}
  }\end{equation} commute, where the upward-pointing arrow is as above. By the uniqueness assertion of the universal property of $\Res^W(\Spec(A))$, one can check that $\varphi'\varphi$ is the identity. By the uniqueness assertion in Lemma \ref{Demazure Envelope Has Adjunction Property for Torsion Free Schemes Algebra Version}, we deduce that $\varphi\varphi'$ is the identity as well. Therefore, (A) holds.

    Observe that (C') implies (C) by taking $W$-invariants. Conversely, (C) implies (C') since base change along a flat morphism preserves torsion-freeness, see for example \cite[Lemma 15.22.4]{StacksProject}. We have thus shown that (C) implies (A). The fact that (A) implies (D) follows from the universal property of $\Res^W(\Spec(A))$ and Lemma \ref{Closed Under Demazure Implies Recoverable from GIT Quotient}.

    Now assume (D), i.e. that the map (\ref{Pull back from Demazure envelope map for affine schemes}) is surjective. Lemma \ref{Demazure Envelope Has Adjunction Property for Torsion Free Schemes Algebra Version} then implies that  (\ref{Pull back from Demazure envelope map for affine schemes}) is bijective. Moreover, we have isomorphisms \begin{equation}\label{Pf iso 1}
     \Maps_{\LC}(\Spec(B), \Spec(E(A)^W)) \xrightarrow{\sim} \Maps_{\LTd}^W (\LTd \times_{\LC} \Spec(B), \Spec(E(A)^W)_{\LTd})\tag{1}
    \end{equation} and \begin{equation}\label{Pf iso 2}
   \Maps_{\LTd}^W (\LTd \times_{\LC} \Spec(B), \Spec(E(A)^W)_{\LTd}) \xrightarrow{\sim}    \Maps_{\LTd}^W (\LTd \times_{\LC} \Spec(B), \Spec(E(A))) \tag{2}
    \end{equation} by Proposition \ref{Fully faithfulness of tensoring up even for algebra objects} and Lemma \ref{Closed Under Demazure Implies Recoverable from GIT Quotient} respectively. Combining the isomorphisms (\ref{Pf iso 1}), (\ref{Pf iso 2}), and (\ref{Pull back from Demazure envelope map for affine schemes}), we deduce an isomorphism \[\Maps_{\LC}(\Spec(B), \Spec(E(A)^W))  \xrightarrow{\sim} \Maps^W_{\LTd}(\LTd \times_{\LC} \Spec(B), \Spec(A))\] which is natural in $\Spec(B)$. Therefore, $\Spec(E(A)^W)$ has the universal property of the Weil restriction, and so our claim follows since $\Res^W(\Spec(A))$ is unique up to unique isomorphism.
\end{proof}

Finally, we prove Corollary \ref{intro corollary for res}. In fact, we prove a slightly more general result:
\begin{Corollary}\label{intro corollary for res expanded}
 For $X$ as in Theorem \ref{intro thm}, there is an isomorphism \begin{equation}\label{Oc algebra iso}\O(\Res^W_{\circ} X) \cong \O(X)^W\Big[\frac{j}{\Delta}\Big]_{j \in \O(X)^{\mathrm{sign}}}\end{equation} of $\O(\LC)$-algebras identifying the ring of functions on $\Res^W_{\circ}(X)$ with the smallest subring of the field of fractions of $\O(X)^W$ containing $\O(X)^W$ together with all elements of the form $\frac{j}{\Delta}$ for $j \in \O(X)^{\mathrm{sign}}$. Moreover, there is an isomorphism \begin{equation}\label{Otstar algebra iso}\O(\Res^W_{\circ} X \times_{\LC} \LTd) \cong \O(X)\Big[\frac{j}{\Delta}\Big]_{j \in \O(X)^{\mathrm{sign}}} \end{equation} of $\O(\LTd)$-algebras. 
 \end{Corollary}
 
\begin{proof}
By Proposition \ref{env prop}(i), we have that \begin{equation}\label{EA is Demazure}E(A) = A\Big[D_{w_0}(a)\Big]_{a \in A}\end{equation} where we have set $A := \O(X)$. We also obtain \begin{equation}\label{Demazure is Sign}A\Big[D_{w_0}(a)\Big]_{a \in A} = A\Big[\frac{a}{\Delta}\Big]_{a \in A^{\mathrm{sign}}}\end{equation} by examining the formula (\ref{Explicit Formula for Dw0}). 

Next, we claim that the natural inclusion \begin{equation}\label{W Invariants of Fixed Points}A^W\Big[\frac{a}{\Delta}\Big]_{a \in A^{\mathrm{sign}}} \subseteq (A\Big[\frac{a}{\Delta}\Big]_{a \in A^{\mathrm{sign}}})^W\end{equation} is an equality. To see this, observe that we may write any element of the right-hand side of (\ref{W Invariants of Fixed Points}) as a sum $\sum_{i, j}b_{i, j}\left(\frac{a_{i, 1}}{\Delta}\right)\cdots\left(\frac{a_{i, j}}{\Delta}\right)$ for some $a_{i, 1}, \ldots, a_{i, j} \in A^{\mathrm{sign}}$ and $b_{i, j} \in A$. 

For each fixed $i, j$, we may write $b_{i, j} = b_{i, j}^{\mathrm{triv}} + b_{i, j}'$ where $b_{i, j}^{\mathrm{triv}} \in A^W$ and $b_{i, j}'$ is a sum of elements in the \textit{nontrivial} isotypic components of $W$. Therefore, $\sum_{i, j}b_{i, j}'\left(\frac{a_{i, 1}}{\Delta}\right)\cdots\left(\frac{a_{i, j}}{\Delta}\right)$ is a sum of elements in the \textit{nontrivial} isotypic components of $W$, and so, since $\sum_{i, j}b_{i, j}\left(\frac{a_{i, 1}}{\Delta}\right)\cdots\left(\frac{a_{i, j}}{\Delta}\right)$ is fixed by the $W$-action, we deduce that $\sum_{i, j}b_{i, j}'\left(\frac{a_{i, 1}}{\Delta}\right)\cdots\left(\frac{a_{i, j}}{\Delta}\right) = 0$. Therefore, \[\sum_{i, j}b_{i, j}\left(\frac{a_{i, 1}}{\Delta}\right)\cdots\left(\frac{a_{i, j}}{\Delta}\right) = \sum_{i, j}b^{\mathrm{triv}}_{i, j}\left(\frac{a_{i, 1}}{\Delta}\right)\cdots\left(\frac{a_{i, j}}{\Delta}\right) \in A^W\Big[\frac{a}{\Delta}\Big]_{a \in A^{\mathrm{sign}}}\] as desired. 

We now obtain \[E(A)^W = A^W\Big[\frac{a}{\Delta}\Big]_{a \in A^{\mathrm{sign}}}\] by taking the $W$-invariants of (\ref{EA is Demazure}) and (\ref{Demazure is Sign}) and using the fact that (\ref{W Invariants of Fixed Points}) is an equality. Since Proposition \ref{ResFreeClosureW is Spec of W Invariants of Demazure Operator} shows that the ring of functions on $\Res^W_{\circ}(X)$ is isomorphic to $E(A)^W$, we deduce (\ref{Oc algebra iso}). 

We now prove the equality (\ref{Otstar algebra iso}). Observe that both sides of this equality are closed under the action of the Demazure operators: indeed, \[\O(\LTd \times_{\LC} \Res_{\circ}^W(X)) \cong \Symt \otimes_{\Symt^W} \O(\Res^W_{\circ}(X))\] is closed under them since $\Symt$ is, and $\O(X)\left[\frac{j}{\Delta}\right]_{j \in \O(X)^{\mathrm{sign}}}$ is closed under them by Proposition \ref{env prop}(i). Therefore, by Lemma \ref{Closed Under Demazure Implies Recoverable from GIT Quotient}, the equality (\ref{Otstar algebra iso}) holds if it holds on $W$-invariants. Since (\ref{W Invariants of Fixed Points}) is an equality, we see that the $W$-invariants of the terms in (\ref{Otstar algebra iso}) are the terms of the equality (\ref{Oc algebra iso}), and (\ref{Otstar algebra iso}) therefore follows from (\ref{Oc algebra iso}). 
\end{proof}

\section{\texorpdfstring{Weil restriction and $\overline{T^*(G/U)}$}{Weil restriction and T*(G/U)}}\label{Weil Restriction and Cotangent Bundle of G mod U Section}
\subsection{}
Let $\flatGSpacesOverLGd$ denote the category of affine schemes $S := \Spec(A)$ equipped with a flat map $S \to \LGd$ and a $G$-action for which this map is $G$-equivariant. In what follows, for any scheme $X$ over $\LGd$, we will use the notation $X_{\mathrm{reg}} := X \times_{\LGd} \LGd_{\mathrm{reg}}$ to denote the preimage of the open subset $\LGd_{\mathrm{reg}}$ of regular elements in $\LGd$. We will also use the notation $\tgd := G \times^{B}(\LG/\mathfrak{u})^*$ for the (dual) Grothendieck-Springer resolution and $\tgdreg$ for its restriction to the regular locus.

Let $\jmath: T^*(G/U)_{\mathrm{reg}} \mathrel{\hookrightarrow} \affineClosureOfCotangentBundleofBasicAffineSpace$ denote the affinization map, which is an open embedding. We recall the birational map \begin{equation*}\varpi: \Tpsirt \to T^*(G/U)_{\mathrm{reg}}\end{equation*} of \cite{GinzburgKazhdanDifferentialOperatorsOnBasicAffineSpaceandtheGelfandGraevAction}, whose composite with $\jmath$ is the map discussed in Section \ref{Whittaker Cotangent Bundle as Weil Restriction Subsection}. 

We now prove Theorem \ref{Representability for Flat G Spaces Over LGd}. In fact, we will prove the following slightly stronger result, which, upon reading the top row of the diagram below, obviously implies Theorem \ref{Representability for Flat G Spaces Over LGd}: 

\begin{Theorem}\label{Souped up Representability for Flat G Spaces Over LGd}\newcommand{\Tpsiliteral}{T^{\psi}}
If the derived subgroup of $G$ is simply connected, then each of the maps in the commutative diagram 

\begin{equation*}\xymatrix@R+2em@C+1em{\Maps_{\LGd}^G(S, \Tpsiliteral) \ar[r]^{\LTd \times_{\LTd\sslash W} -} \ar[d]^{\LGd_{\mathrm{reg}} \times_{\LGd} -} & \Hom^{G \times W}_{\LGd \times_{\mathfrak{c}} \LTd}(S_{\LTd}, \Tpsiliteral_{\LTd}) \ar[d]^{\LGd_{\mathrm{reg}} \times_{\LGd} -} \ar[r]^{\jmath\varpi \circ -} &  \Maps_{\LGd \times_{\mathfrak{c}} \LTd}^{G \times W}(S_{\LTd}, \affineClosureOfCotangentBundleofBasicAffineSpace)\ar[d]^{\LGd_{\mathrm{reg}} \times_{\LGd} -}  \\
\Hom^{G}_{\LGd_{\mathrm{reg}}}(S_{\mathrm{reg}}, \Tpsiliteral)  \ar[r]^{\LTd \times_{\LTd\sslash W} -} & \Hom^{G \times W}_{\LGd_{\mathrm{reg}} \times_{\mathfrak{c}} \LTd}((S_{\mathrm{reg}})_{\LTd}, \Tpsiliteral_{\LTd})  \ar[r]^{\varpi \circ -}& \Maps_{\LGd_{\mathrm{reg}} \times_{\mathfrak{c}} \LTd}^{G \times W}((S_{\mathrm{reg}})_{\LTd}, T^*(G/U)_{\mathrm{reg}})  
 }\end{equation*} is an isomorphism for any $S \in \flatGSpacesOverLGd$, where we use the shorthand $\Tpsiliteral := \Tpsir$.
\end{Theorem}

\subsection{\texorpdfstring{Inverse of restriction to the regular locus}{Inverse of restriction to the regular locus}}\label{Affinization is isomorphism for Flat LGd Schemes} We now prove:

\begin{Lemma}\label{Extendability of Flat Affine Schemes}
Let $S_1$ and $S_2$ be affine schemes, each equipped with a map to an irreducible affine finite-type $k$-scheme $B$ satisfying Serre's condition $(S_2)$.\footnote{Note that this is automatically satisfied if $B$ is normal; in particular, it is satisfied if $B$ is smooth.} Assume further that $S_1 \to B$ is flat. For any open subset $B_0 \subseteq B$ whose complement has codimension at least two, the restriction map \[\Maps_{B}(S_1, S_2) \mathrel{\to} \Maps_{B_0}(S_1 \times_B B_0, S_2 \times_B B_0)\] is an isomorphism.
\end{Lemma}

\begin{proof}
Let $f_1: S_1 \to B$ and $f_2: S_2 \to B$ denote the respective structure maps. Our assumption that $f_1$ is flat implies that $f_{1*}(\O_{S_1})$ is flat as a quasicoherent sheaf on $B$. Therefore, by Lazard's theorem \cite[Theorem 10.81.4]{StacksProject}, $f_{1*}(\O_{S_1})$ is a filtered colimit of free $\O_B$-modules. Therefore, if $j: B_0 \mathrel{\hookrightarrow} B$ denotes the open embedding, we have an isomorphism \[f_{1*}(\O_{S_1}) \xrightarrow{\sim} j_*j^*f_{1*}(\O_{S_1})\] given by the unit map of the adjunction; here we use that $B$ is of finite type, so that $j_*$ commutes with filtered colimits, and that $B$ satisfies $(S_2)$, so that the natural map $\O_B \mathrel{\to} j_*j^*(\O_B)$ is an isomorphism. 

Now, naturality implies that, for any $\O_B$-algebra map $\phi: f_{2*}(\O_{S_2}) \to f_{1*}(\O_{S_1})$, the diagram \begin{equation}\xymatrix@R+2em@C+2em{ f_{2*}(\O_{S_2}) \ar[r]^{\phi} \ar[d]  & f_{1*}(\O_{S_1}) \ar[d]^{\sim}\\
j_*j^*f_{2*}(\O_{S_2}) \ar[r]^{j_*j^*(\phi)} & j_*j^*f_{1*}(\O_{S_1}) 
  }\end{equation} commutes. We immediately obtain injectivity. Finally, given an $\O_{B_0}$-algebra map \[\psi: j^*f_{2*}(\O_{S_2}) \to j^*f_{1*}(\O_{S_1}),\] the composite \[f_{2*}(\O_{S_2}) \to j_*j^*f_{2*}(\O_{S_2}) \xrightarrow{j_*\psi} j_*j^*f_{1*}(\O_{S_1}) \xrightarrow{\sim} f_{1*}(\O_{S_1})\] restricts to $\psi$, by the natural isomorphism $j^*j_* \xrightarrow{\sim} \mathrm{id}$. This proves surjectivity.
\end{proof}

Observe that, given a map to a separated scheme equipped with the action of some group, we may check equivariance on a group-stable schematically dense open subset of the source. In our case, $S$ and $S_{\LTd}$ are flat over $\LGd$, so the inverse images of $\LGdreg$ are schematically dense. Compatibility with the maps to $\LTd$ can be checked in the same way. Moreover, the complement of $\LGdreg$ has codimension at least two in $\LGd$. To apply Lemma \ref{Extendability of Flat Affine Schemes} to the rightmost vertical arrow, we use the identification \[\affineClosureOfCotangentBundleofBasicAffineSpace \times_{\LGd} \LGdreg \cong T^*(G/U)_{\mathrm{reg}}.\] Indeed, $T^*(G/U)_{\mathrm{reg}}$ is affine over $\LGdreg$, since it is a $T$-torsor over $\tgdreg$ and $\tgdreg \to \LGdreg$ is finite. The identification then follows by localization of global functions on principal open subsets of $\LGd$ contained in $\LGdreg$. Thus Lemma \ref{Extendability of Flat Affine Schemes} proves that all of the vertical arrows in the diagram of Theorem \ref{Souped up Representability for Flat G Spaces Over LGd} are isomorphisms. 

\subsection{\texorpdfstring{Inverse of base change}{Inverse of base change}}\label{Inverse of Base Change Subsubsection} The fact that the top left rightward-pointing arrow is an isomorphism follows immediately from Proposition \ref{Fully faithfulness of tensoring up even for algebra objects}. One can similarly construct the inverse to the bottom left rightward-pointing arrow in the diagram of Theorem \ref{Souped up Representability for Flat G Spaces Over LGd} by taking the categorical quotient $- \sslash W$. Alternatively, since the leftmost square in the diagram of Theorem \ref{Souped up Representability for Flat G Spaces Over LGd} commutes and its vertical arrows and top horizontal arrow are isomorphisms, we immediately deduce that the bottom left horizontal arrow is an isomorphism.
\subsection{\texorpdfstring{The pullback map is an isomorphism on the regular locus}{The pullback map is an isomorphism on the regular locus}} We now complete the proof of Theorem \ref{Souped up Representability for Flat G Spaces Over LGd}. In Section \ref{Affinization is isomorphism for Flat LGd Schemes}, we saw that all vertical maps in the diagram of Theorem \ref{Souped up Representability for Flat G Spaces Over LGd} are isomorphisms. In Section \ref{Inverse of Base Change Subsubsection}, we showed that the two leftmost horizontal arrows are isomorphisms. Therefore Theorem \ref{Souped up Representability for Flat G Spaces Over LGd} follows from the following claim, which says that the lower right horizontal arrow is an isomorphism even if $S$ is not flat over $\LGd$:

\begin{Proposition}
If $G$ does not contain $\SO_{2n + 1}$ as a direct factor, then the map $\varpi$ induces an isomorphism \[\Tpsir \xrightarrow{\sim} \Res^W(T^*(G/U)_{\mathrm{reg}}).\]
\end{Proposition}

\newcommand{\cotangentBundleOfBaseAffineSpaceREGULARPOINTS}{T^*(G/U)_{\mathrm{reg}}}

\begin{proof}
Recall that, by \cite[Proposition 5.2.5]{GinzburgKazhdanDifferentialOperatorsOnBasicAffineSpaceandtheGelfandGraevAction}, there is a $W$-equivariant map \[\Tpsir \times_{\LC} T^*T \to T^*(G/U)_{\mathrm{reg}}\] inducing an isomorphism \[\Tpsir \times^{J_G}_{\LC} T^*T \xrightarrow{\sim} T^*(G/U)_{\mathrm{reg}}\] on the balanced product. Let $q: \tgdreg = \LGdreg \times_{\LC} \LTd \to \LGdreg$ denote the projection map. The $W$-invariant moment map allows us to identify $\Res^W_{\LTd \to \LC}(T^*(G/U)_{\mathrm{reg}})$, viewed as a scheme over $\LGdreg$, with $\Res_q^W(T^*(G/U)_{\mathrm{reg}})$. By adjunction, $\varpi$ induces a map $\epsilon: \Tpsir \to \Res^W_q(T^*(G/U)_{\mathrm{reg}})$.

By \cite[Lemma 3.2.3(ii)]{GiKa},  the map $\Tpsir \to \LGdreg$ is an
  \'etale-locally trivial torsor under the pullback of $J_G$.  Hence,  there exists an \'etale  cover $\mathbf{U} \to \LGdreg$ over which this torsor is trivial. Choose a section over $\mathbf{U}$, and set $\widetilde{\mathbf{U}} := \mathbf{U} \times_{\LC} \LTd$. The balanced product description above then gives a $W$-equivariant isomorphism \[T^*(G/U)_{\mathrm{reg}} \times_{\LGdreg} \mathbf{U} \cong \widetilde{\mathbf{U}} \times_{\LTd} T^*T.\] Applying the $W$-invariant Weil restriction and using base change (Proposition \ref{Base Change of Weil Restriction}), we obtain an isomorphism \[\Res^W_q(T^*(G/U)_{\mathrm{reg}}) \times_{\LGdreg} \mathbf{U} \cong \mathbf{U} \times_{\LC} \Res^W(T^*T).\] Under these identifications and the chosen trivialization $\Tpsir \times_{\LGdreg} \mathbf{U} \cong \mathbf{U} \times_{\LC} J_G$, the map $\epsilon_{\mathbf{U}}$ is the base change of the canonical map $J_G \to \Res^W(T^*T)$. Our hypothesis and (\ref{Isomorphism of Ngo}) imply that this map is an isomorphism. Therefore $\epsilon_{\mathbf{U}}$ is an isomorphism and, since $\mathbf{U} \to \LGdreg$ is an
\'etale cover, so too is $\epsilon$, as desired.
\end{proof}

\printbibliography
\end{document}